\documentclass[11pt]{article}
\usepackage{amssymb,amsmath}
\usepackage{xcolor}
\usepackage{pgfplots}
\date{\mbox{}}
\newtheorem{prop}{Proposition}
\newtheorem{lemma}{Lemma}
\newtheorem{defi}{Definition}

\newtheorem{cor}{Corollary}
\newtheorem{rem}{Remark}
\newtheorem{ex}{Example}

\newenvironment{proof}{\noindent Proof}

\author{Abbas Ali and Assi Abdallah
\footnote{LAREMA, Angers university, 2 Bd Lavoisier, 49045 Angers cedex 01, emails: ali.abbas@math.univ-angers.fr and assi@math.univ-angers.fr\newline
}}

\begin{document}
\title{Semigroup associated with a free polynomial}
\maketitle
\begin{abstract}
    Let  $\mathbb{K}$  be an algebraically closed field of characteristic zero and let  ${\mathbb K}_{C}[\![x_{1},\cdots,x_{e}]\!]$ be the ring of formal power series in several variables with  exponents in a line free cone $C$. We consider irreducible polynomials $f=y^n+a_1(\underline{x})y^{n-1}+\cdots+a_n(\underline{x})$ in ${\mathbb K}_{C}[\![x_{1},\cdots,x_{e}]\!][y]$ whose roots are in ${\mathbb K}_{C}[\![x_{1}^{\frac{1}{n}},\cdots,x_{e}^{\frac{1}{n}}]\!]$. We generalize to these polynomials the theory of Abhyankar-Moh. In particular we associate with any such polynomial its set of characteristic exponents and its semigroup of values. We also prove that the set of values can be obtained using the set of approximate roots. We finally prove that  polynomials of ${\mathbb K}[\![x_1,\cdots,x_e]\!][y]$ fit in the above set for a specific line free cone (see Section 4). 
\end{abstract}
\section*{Introduction}
Let ${\mathbb K}$ be an algebraically closed field of characteristic zero and let ${\mathbb K}[\![\underline{x}]\!]$ be the ring of formal power series in $\underline{x}=(x_1,\cdots,x_e)$ over ${\mathbb K}$. Let $f=y^{n}+a_{1}(\underline{x})y^{n-1}+\cdots +a_{n}(\underline{x})$ be a nonzero polynomial of degree $n$ in ${\mathbb K}[\![\underline{x}]\!][y]$. Suppose that $f$ is   a  quasi-ordinary polynomial, i.e its discriminant $\Delta_{y}(f)$(the $y$-resultant of $f$ and its $y$-derivative), is of the form $\Delta_{y}(f)=\underline{x}^{\underline \alpha}.\varepsilon(\underline{x})$, where $\varepsilon(\underline{x})$ is a unit in $\mathbb{K}[\![\underline{x}]\!]$ (Note that this is always the case if $e=1$). If $f$ is irreducible then, by the Abhyankar-Jung theorem, there exists $y=\sum_{p\in{\mathbb N}^e}c_{p}\underline{x}^{\frac{p}{n}}\in \mathbb{K}[\![\underline{x}^{\frac{1}{n}}]\!]$ such that  $f(\underline{x},y)=0$. Define the support of $y$ to be the set Supp$(y)=\{\frac{p}{n} \mid  c_{p}\neq 0\}$. In \cite{Lipman1}, Lipman  proved that there exists a sequence of elements $\frac{m_1}{n},\cdots,\frac{m_h}{n}\in {\rm Supp}(y)$ such that:

\noindent $(i)$\ $m_{1}<m_{2}<\cdots <m_{h}$ coordinate wise.\\
$(ii)$\  If $\frac{m}{n}\in {\rm Supp}(y)$, then $m\in (n\mathbb{Z})^{e}+\displaystyle\sum_{i=1}^{h}m_{i}\mathbb{Z}$. Moreover,  $m_{i}\notin (n\mathbb{Z})^{e}+\sum_{j<i}m_{j}\mathbb{Z}$ for all $i=1,\cdots,h$.\\
\noindent The semigroup of $f$ is defined to be the set $\Gamma(f)=\{O(f,g),\ g\in{\mathbb K}[\![\underline{x}]\!][y]\backslash(f)\}$, where $O(f,g)$ is the  order of the initial form of the $y$-resultant of $f$ and $g$ with respect to a fixed order on ${\mathbb N}^e$ (we also have  $O(f,g)=n O(g(\underline{x},y(\underline{x}))$ where the latter $O$ denotes the leading monomial of the series $g(\underline{x},y(\underline{x})))$. Now we can associate with $f$ the following sequences: the $\underline{D}$-sequence of $f$ is defined to be $D_{1}=n^{e}$, and for all $2\leq i\leq h$, $D_{i}$ is the gcd of the $(e,e)$ minors of the matrix $[nI_{e},m_{1}^{T},\cdots,m_{i-1}^{T}]$, where $T$ denotes the transpose of a vector. We have  $D_{1}>...>D_{h+1}=n^{e-1}$. Then we define the $\underline{e}$-sequence to be $e_{i}=\frac{D_{i}}{D_{i+1}}$ for all $1\leq i\leq h$, and the $\underline{r}$-sequence $ r_{0}^{1},\cdots,r_{0}^{e},r_{1},\cdots,r_{h}$ to be $r_1=m_1,
r_{i}=e_{i-1}r_{i-1}+m_{i}-m_{i-1}
$
for all $2\leq i\leq h$, and $r_0^1,\cdots,r_0^e$ is the canonical basis of ${n\mathbb Z}^e$. The sequence $\lbrace r_{0}^{1},\cdots,r_{0}^{e},r_{1},\cdots,r_{h}\rbrace$ is a system of generators of  $\Gamma(f)$. Moreover, there exists a special set of polynomials $g_{1},\cdots,g_{h}$ (the approximate roots of $f$), such that $O(f,g_{i})=r_{i}$ for all $i\in\{1,\cdots,h\}$( see \cite{A.Assi1}). \\ 
The aim of this article is to generalize these results to a wider class of polynomials. Namely let $C$ be a line free rational convex cone in ${\mathbb R}^e$ and let ${\mathbb K}_C[\![\underline{x}]\!]$ be the ring of power series whose exponents are in $C$. Let  $f=y^{n}+a_{1}(\underline{x})y^{n-1}+\cdots +a_{n}(\underline{x})$ be a nonzero polynomial of $\mathbb{K}_C[\![\underline{x}]\!][y]$. We say that $f$ is free if it is irreducible in  $\mathbb{K}_C[\![\underline{x}]\!][y]$ and if it has a root (then all its roots) $y(\underline{x})\in {\mathbb K}_C[\![{\underline{x}}^{\frac{1}{n}}]\!]$. Note that irreducible quasi-ordinary polynomials are free with respect to the cone ${\mathbb R}_+^e$. Then we associate with a free polynomial $f$ its set of characteristic exponents and characteristic sequences. We also associate with $f$ its set of pseudo-approximate roots  and we prove that the set of orders (with respect to a fixed order on ${\mathbb Z}^e\cap C$) of these polynomials generate the semigroup of $f$, which is defined to be the set of orders of polynomials of ${\mathbb{K}}_C[\![\underline{x}]\!][y]$ (see Definition 7). Finally  we prove that the semigroup is also generated by the set of orders of approximate roots of $f$ (see Section 3). Note that the semigroup is free in the sense of \cite{A.Assi2}. This explains the notion of free polynomials (see Remark \ref{free} for more details). In Section 4 we apply our results to polynomials of ${\mathbb K}[\![\underline{x}]\!][y]={\mathbb K}_{{\mathbb R}_+^e}[\![\underline{x}]\!][y]$. An irreducible polynomial $f\in {\mathbb K}[\![\underline{x}]\!][y]$ is not free in general. Our main result is that $f$ becomes free in $\mathbb{K}_C[\![\underline{x}]\!][y]$ for a specific cone, after a preparation result. More precisely let $\Delta_{y}(f)$ be the $y$-discriminant of $f$. If $f$ is a prepared polynomial (in the sense of Remark 4) then $f$ is equivalent, modulo a birational transformation, to a quasi-ordinary polynomial $F$. This transformation is used in order to go from roots of $F$ to roots of $f$, and these roots are in ${\mathbb K}_C[\![\underline{x}^{\frac{1}{n}}]\!]$ for the cone introduced in Proposition \ref{conecone}.  

\noindent We would like to point out that our results generalize those of J.T. Tornero in \cite{Tornero} where polynomials are free (but not necessarily quasi-ordinary) in the cone ${\mathbb R}_+^e$. 

\section{G-adic expansion and Approximate roots}
\noindent  In this section we recall the notion of $G$- adic expansion and the notion of approximate roots (see \cite{Abyankar1}). Let $R[Y]$ be the polynomial ring in one variable over an integral domain $R$.
\begin{prop}
Let $f$ be a polynomial of degree $n$ in $R[Y]$ and let $d$ be a divisor of $n$. Let $g$ be a monic polynomial of degree $\frac{n}{d}$, then there exist unique polynomials $a_{1},\cdots,a_{d}\in R[Y]$ with $deg_{Y}(a_{i})<\frac{n}{d}$ for all $i\in\{1,\cdots,d\}$ such that $a_i\not=0$, and
$f=g^{d}+a_{1}g^{d-1}+\cdots+a_{d}.$
\end{prop}
\noindent This expression is called the $g$\textcolor{red}{-}adic expansion of $f$. The \textbf{Tschirnhausen transform} of $g$ with respect to $f$ is defined to be 
$\tau_{f}(g)=g+d^{-1}a_{1}$. Note that  $\tau_f(g)$ is a monic polynomial of degree $\frac{n}{d}$ and so we can define recursively the $i^{th}$ Tschirnhausen transform of $g$ to be $\tau_{f}^{i}(g)=\tau_{f}(\tau_{f}^{(i-1)}(g))$ with $\tau_{f}^{1}(g)=\tau_{f}(g)$. By $\cite{Abyankar1}$, $\tau_{f}(g)=g$ if and only if $a_1=0$ if and only if $deg(f-g^{d})<n-\frac{n}{d}$. In this case $g$ is  said to be the $d^{th}$ approximate root of $f$. For every divisor $d$ of $n$  there exists a unique $d^{th}$ approximate root of $f$. We denote it by $App(f,d)$.
\medskip

\noindent More generally let $n=d_{1}>d_{2}>...>d_{h}$ be a sequence of integers such that $d_{i+1}$ divides $d_{i}$ for all $i\in\lbrace 1,\cdots,h-1\rbrace$, and set $e_{i}=\frac{d_{i}}{d_{i+1}}, 1\leq i\leq h-1$, and $e_{h}=+\infty$. For all $i\in\{1,\cdots,h\}$ let $G_{i}$ be a monic polynomial of degree $\frac{n}{d_{i}}$ (in particular deg$_YG_1=1$) and let $G=(G_1,\cdots,G_h)$. 
Let $B=\{\underline{b}=(b_{1},\cdots,b_{h})\in\mathbb{N}^{h},\ 0\leq b_{i}<e_{i}\ \forall 1\leq i\leq h \}$. Then $f$ can be written in a unique way as $f=\sum_{{\underline{b}}\in B}c_{\underline{b}}G_{1}^{b_{1}}\cdots G_{h}^{b_{h}}$. We call this expression the $G$-adic expansion of $f$.


\section{Line Free Cones}
In this section we recall the notion of line free cones, which will be used later in the paper. Let $C\subseteq \mathbb{R}^{e}$. We say that $C$ is  a cone if for all $s\in C$ and for all $\lambda \geq 0$, $\lambda s\in C$. A cone $C$ is said to be finitely generated  if there exists a finite subset $\{s_{1},\cdots,s_{k}\}$ of $C$ such that for all $s\in C$, 
$$s=\lambda_{1}s_{1}+\cdots+\lambda_{k}s_{k}$$
for some $\lambda_{1},\cdots,\lambda_{k}\in \mathbb{R}$. If $s_1,\cdots,s_k$ can be chosen to be in $\mathbb{Q}^{e}$, then $C$ is said to be rational. From now on we suppose that all considered cones are  finitely generated and rational.

\begin{defi}
Let $C$ be a (finitely generated, rational) cone, then $C$ is said to be a line free cone if $\forall v\in C-\{0\}$, $-v\notin C$.
\end{defi}


\noindent Given a line free cone, we can define the set of formal power series in several variables with exponents in $C$, denoted  $\mathbb{K}_{C}[\![\underline{x}]\!]$. More precisely an element $y\in\mathbb{K}_{C}[\![\underline{x}]\!]$ is of the form $y=\sum_{p=(p_{1},\cdots,p_{e})\in C\cap{\mathbb Z}^e}\alpha_{p}x_{1}^{p_{1}}\cdots x_{e}^{p_{e}}$. It follows from \cite{Monforte} that  this set is a ring. 

\begin{defi}
Let $\leq$ be a total order on $\mathbb{Z}^{e}$, then $\leq $ is said to be additive if for all $m,n,k\in \mathbb{Z}^{e}$ we have : $m\leq n \implies m+k\leq n+k$. An additive order on $\mathbb{Z}^{e}$ is said to be compatible with a cone $C$ if $m\geq 0=(0,\cdots,0)$ for all $m\in C\cap \mathbb{Z}^{e}$.
\end{defi}

\noindent With these notations we have the following:

\begin{prop}\label{compatibleorder}
(see \cite{Monforte}) Let $C$ be a line free cone. There exists an additive total order $\leq$ which is compatible with $C$. Moreover, if $\leq$ is such a total order, then $\leq$ is a well-founded order on $C\cap \mathbb{Z}^{e}$, i.e, every subset of $C\cap \mathbb{Z}^{e}$ contains a minimal element with respect to the chosen order, and this minimal element is unique.
\end{prop}


\noindent Let $y=\sum_{p}c_{p}\underline{x}^{p}$ be an element in $\mathbb{K}_{C}[\![\underline{x}]\!]$. The support of $y$, denoted Supp$(y)$, is defined to be the set of elements $p\in C$ such that $c_{p}\neq 0$. It results  from Proposition $\ref{compatibleorder}$  that elements in Supp$(y)$ can be written as an increasing sequence with respect to the chosen additive order on $C$.

\noindent We shall now introduce the notion of free polynomials.
\begin{defi}

Let $C$ be a line free  cone and let $f=y^n+a_1(\underline{x})y^{n-1}+\cdots+a_n(\underline{x})\in \mathbb{K}_{C}[\![\underline{x}]\!][y]$. Then $f$ is said to be a free polynomial if $f$ is irreducible in $\mathbb{K}_{C}[\![\underline{x}]\!][y]$ and if it has a root $y(\underline{x})$ in $\mathbb{K}_{C}[\![\underline{x}^{\frac{1}{n}}]\!]$.
\end{defi}

\begin{rem} We may have many choices for the total order in Proposition \ref{compatibleorder}. For example, let $C={\mathbb R}_+^e$ and let $y(x_1,x_2)=x_1+x_2$, then we can arrange Supp$(y)$ by either $(1,0)<(0,1)$ or $(0,1)<(1,0)$, depending of chosen order on $C$.
\end{rem}

\section{Characteristic sequences of a free polynomial}

In this section we will introduce the set of   characteristic sequences associated with a free polynomial as well as its semigroup. Let $C$ be a (finitely generated, rational) line free cone and let $\leq$ be an additive order  on ${\mathbb Z}^e$ compatible with $C$. Let $f=y^n+a_1(\underline{x})y^{n-1}+\cdots+a_n(\underline{x})\in\mathbb{K}_{C}[\![\underline{x}]\!][y]$  be a free  polynomial and let $y=\sum c_{p}x^{\frac{p}{n}}\in\mathbb{K}_{C}[\![\underline{x}^{\frac{1}{n}}]\!]$ be a root of $f$. Let $L$ be the field of fractions of $\mathbb{K}_{C}[\![\underline{x}]\!]$ and  set
$L_{1}=L(x_{1}^{\frac{1}{n}}),L_{2}=L_{1}(x_{2}^{\frac{1}{n}}),\cdots,L_{n}=L_{n-1}(x_{e}^{\frac{1}{n}})=L(x_{1}^{\frac{1}{n}},\cdots,x_{e}^{\frac{1}{n}})$. Then $L_{n}$ is a galois extension of $L$ of degree $n^{e}$.
Let finally $U_{n}$ be the set of $n^{th}$ roots of unity in $\mathbb{K}$. 

\noindent Let 
$\theta\in Aut(L_{n}/L)$. For all $i\in \lbrace 1,\cdots,e\rbrace$ we have 
 $\theta(x_{i}^{\frac{1}{n}})=\omega_{i} x_{i}^{\frac{1}{n}}$ for some $\omega_{i}\in U_{n}$. Then 
$\theta(\underline{x}^{\frac{p}{n}})=k\underline{x}^{\frac{p}{n}}$, 
where $k$ is a non zero element of $\mathbb{K}$. Let $Roots(f)=\{y_{1},\cdots,y_n\}$ be the conjugates of $y$ over $L$, with the assumption that $y_{1}=y=\sum c_{p}\underline{x}^{\frac{p}{n}}$. Then for all $2\leq i\leq n$ there exists an automorphism $\theta\in Aut(L_{n}/L)$ such that $y_{i}=\theta(y)$, hence $y_{i}=\theta(y)=\sum c_{p}k_{p}\underline{x}^{\frac{p}{n}},\ k_{p}\in\mathbb{K}^{*}$, and consequently  Supp$(y)={\rm Supp}(y_{i})$.\\
Let $z\in\mathbb{K}_{C}[\![\underline{x}^{\frac{1}{n}}]\!]$. Then $n{\rm Supp}(z)=\lbrace k\mid \frac{k}{n}\in{\rm Supp}(z)\rbrace$ can be arranged into an increasing sequence with respect to $\leq$. We define the order of $z$, denoted $O(z)$, to be  $O(z)=\dfrac{1}{n}{\rm inf}_{\leq}n{\rm Supp}(z)$ if $z\neq 0$, and $O(0)=+\infty$. We set $LM(z)=\underline{x}^{\frac{p}{n}}$ where $\frac{p}{n}=O(z)$, and we call it the leading monomial of $z$. We set $LC(z)$ the coefficient of $x^{O(z)}$ and we call it the leading coefficient of $z$. We finally set $Info(z)=LC(z)LM(z)$ and we call it the initial form of $z$.


\begin{defi}
Let the notations be as above with  $\{y_{1},\cdots,y_{n}\}=Roots(f)$. The set of characteristic exponents of $f$ is defined to be $ \{m_{ij}=nO(y_{i}-y_{j}) \mid 1\leq i\not=j\leq n \}$. Similarly we define the set of characteristic monomials of $f$ to be $\{LM(y_{i}-y_{j}) \mid 1\leq i\not=j\leq n\}$.
\end{defi}

\noindent Next we will give some properties of the set of characteristic exponents.

\begin{prop}\label{27}
Let the notations be as above. Then the set of characteristic exponents of $f$ is equal to the set 
$\{nO(y_{k}-y_1) \mid 2\leq k\leq n\}$.
In particular the set of characteristic monomials of 
$f$
is given by 
$\{LM(y_{k}-y_1) \mid 2\leq k\leq n\}=\{LM(\theta(y_1)-y_1),\theta(y_1)\neq y_1,\ \theta\in Aut(L_{n}/L)\}$.
\end{prop}
\begin{proof}{.} We only need to prove that any characteristic exponent $m_{ij}$ satisfies $\frac{m_{ij}}{n}=O(y_k-y_1)$ for some $k$.  Let $1\leq i\neq j\leq n$ and let $c_{ij}=LC(y_{i}-y_{j})$ and $M_{ij}=LM(y_{i}-y_{j})$, then
$y_{i}-y_{j}=c_{ij}M_{ij}+\epsilon_{ij}$
where $\epsilon_{ij}\in L_{n}$ and $O(\epsilon_{ij})>O(M_{ij})$. Let $\theta\in Aut(L_{n}/L)$, such that  $\theta(y_{j})=y_1$, then $\theta(y_{i})=y_{k}$ for some $1\leq k\leq n$, and 
$\theta(y_{i}-y_{j})=\theta(y_{i})-\theta(y_{j})=y_{k}-y_1=c_{k1}M_{k1}+\epsilon_{k1}=\theta(c_{ij}M_{ij}+\epsilon_{ij})=c_{ij}\alpha M_{ij}+\theta(\epsilon_{ij})
$ with  $\alpha\neq 0$, $O(\epsilon_{k1}) > O(M_{k1})$, and $O(\theta(\epsilon_{ij}))>O(M_{ij})$. Hence $M_{k1}=M_{ij}=LM(y_{i}-y_{j})$. This proves our assertion.$\blacksquare$
\end{proof}

\noindent Let $\lbrace M_1,\cdots,M_h\rbrace$ be the set of characteristic monomials of $f$ and write $M_i={\underline{x}}^{m_i\over n}$. Then $\lbrace m_1,\cdots, m_h\rbrace$ is the set of characteristic exponents of $f$. We shall suppose that  $m_{1}<m_{2}<...<m_{h}$.  If  $m<m'$ and $N\in{\mathbb N}$ then we shall sometimes write, by abuse of notation,  $\frac{m}{N}<\frac{m'}{N}$, and $\underline{x}^{\frac{m}{N}}<\underline{x}^{\frac{m'}{N}}$.

\begin{prop} Let the notations be as above. We have $L(y_1)=L(M_{1},\cdots,M_{h})$.
\end{prop}

\begin{proof}{.} Let $\theta\in Aut(L_{n}/L(y_1))$, then $\theta$ is an $L$-automorphism of $L_{n}$
 with $\theta(y_1)=y_1$. We have  $\theta(y_1)=\theta(\sum c_{p}\underline{x}^{\frac{p}{n}})=
 \sum c_{p}\theta(\underline{x}^{\frac{p}{n}})=
 \sum c_{p}k_{p}\underline{x}^{\frac{p}{n}}=y_1=\sum c_{p}\underline{x}^{\frac{p}{n}}$, with 
  $k_{p}\neq 0$ for  all $\frac{p}{n}\in {\rm Supp}(y_1),$ and so $ \theta(\underline{x}^{\frac{p}{n}})=\underline{x}^{\frac{p}{n}} $. Hence $\underline{x}^{\frac{p}{n}} \in L(y_1)$ for all  $\frac{p}{n} \in {\rm Supp}(y_1)$. In particular, since $M_{1},\cdots,M_{h}$ are monomials of $y_1$, then $M_{1},\cdots,M_{h}\in L(y_1)$, and so $L(M_{1},\cdots,M_{h})\subset L(y_1)$. Conversely, if $\theta\in Aut(L_{n}/L(M_{1},\cdots,M_{h}))$, i.e if $\theta$ is an $L$ automorphism of $L_{n}$ such that $\theta(M_{i})=M_{i}\ \forall\ i=1,\cdots,h $, then $\theta(y_1)=y_1$. In fact if $\theta(y_1) \neq y_1 $ then $\theta(y_1)-y_1=cM_{i}+\epsilon_{i}$ for some characteristic monomial $M_{i}$, hence $\theta(M_{i})\neq M_{i}$ which contradicts the hypothesis. This proves our assertion.$\blacksquare$
\end{proof}
\noindent Note that for all $i\in\{1,\cdots,h\}$,  $L(M_{1},\cdots,M_{i})=L[M_{1},\cdots,M_{i}]$ since $M_{i}$  is algebraic  over $L$.

\begin{prop}\label{belongstogroup}
Let the notations be as above. If $\frac{m}{n}\in {\rm Supp}(y_1)$ then  $m\in (n\mathbb{Z})^{e}+\sum_{i=1}^{h}m_{i}\mathbb{Z}$.
\end{prop}
\begin{proof}{.} Write $M=\underline{x}^{\frac{m}{n}}$. Since $M$ is a  monomial of $y_1$, then $M\in L(y)=L[M_{1},\cdots,M_{h}] $, hence 
$M=\frac{f_{1}}{g_{1}}M_{1}^{\alpha^{1}_{1}}\cdots M_{h}^{\alpha_{h}^{1}}+\cdots +\frac{f_{l}}{g_{l}}M_{1}^{\alpha^{l}_{1}}
\cdots M_{h}^{\alpha_{h}^{l}}$
for some $f_{1},\cdots,f_{l},g_{1},\cdots,g_{l}\in\mathbb{K}_{C}[\![\underline{x}]\!]$
 and $l\in \mathbb{N}^{*}$, and so
$g_{1}\cdots g_{l}M=f_{1}g_{2}\cdots g_{l}M_{1}^{\alpha_{1}^{1}}\cdots M_{h}^{\alpha_{h}^{1}}+\cdots +f_{l}g_{1}\cdots g_{l-1}M_{1}^{\alpha_{1}^{l}}\cdots M_{h}^{\alpha_{h}^{l}}$.
Comparing both sides we get that $\underline{x}^{b}{M}=LM(g_{1}\cdots g_{l}M)=\underline{x}^{a}M_{1}^{\alpha_{1}^{i}}\cdots M_{h}^{\alpha_{h}^{i}}$ for some $i\in\{1,\cdots,l\}$ and $a,b\in\mathbb{Z}^{e}$. In particular $nb+m=na+\alpha_{1}^{i}m_{1}+...+\alpha_{h}^{i}m_{h}$, and so $m=n(a-b)+\alpha_{1}^{i}m_{1}+...+\alpha_{h}^{i}m_{h}\in(n\mathbb{Z})^{e}+\sum_{i=1}^{h}m_{i}\mathbb{Z}$.$\blacksquare$

\end{proof}

\begin{rem}Write $F_{0}=L$ and for all $i\in\lbrace 1,\cdots,h\rbrace,F_{i}=L[M_{1},\cdots,M_{i}]=F_{i-1}[M_{i}]$. Also let $G_{0}=(n\mathbb{Z})^{e}$ and for all $i\in \lbrace 1,\cdots,h\rbrace, G_{i}=(n\mathbb{Z})^{e}+\sum_{j=1}^{i}m_{j} \mathbb{Z}$. As in  Proposition $\ref{belongstogroup}$, we can prove that for any monomial $M=\underline{x}^{\frac{m}{n}}$ with $m\in C$, we have $M\in F_i\Leftrightarrow m\in G_{i}$.
\end{rem}

\noindent Next we will define the set of characteristic sequences associated with $f$.

\begin{defi}\label{mrd} Let the notations be as above  and let $\{m_{1},\cdots,m_{h}\}$ be the set of characteristic exponents of $f$. Let $I_e$ be the $e\times e$ identity matrix. We shall introduce  the following sequences:\\
$\bullet$ The $GCD$-sequence $\{D_{i}\}_{1\leq i\leq h+1}$, where $D_{1}=n^{e}$ and for all $i\in\{1,\cdots,h\}$,  $D_{i+1}=gcd(nI_{e},m_{1}^{T},.$\\$..,m_{i}^{T})$, the $gcd$ of the $ (e,e)$ minors of the $e\times (e+i)$ matrix $(nI_{e},m_{1}^{T},\cdots,m_{i}^{T})$.\\
$\bullet$ The $d$-sequence $\{d_{i}\}_{1\leq i\leq h+1}$, where $d_{i}=\frac{D_{i}}{D_{h+1}}$.\\
$\bullet$ The $e$-sequence $\{e_{i}\}_{1\leq i\leq h}$, where  $e_{i}=\frac{D_{i}}{D_{i+1}}=\frac{d_{i}}{d_{i+1}}$.\\
$\bullet$ The $r$-sequence $\{r_{0}^{1},\cdots,r_{0}^{e},r_{1},\cdots,r_{h}\}$, where $(r_{0}^{1},...r_{0}^{e})$ is the canonical basis of $(n\mathbb{Z})^{e}$, $r_{1}=m_{1}$, and for all $i\in\{2,\cdots,h\}$  $r_{i}=e_{i-1} r_{i-1}+m_{i}-m_{i-1}$. Note that for all $i\in\{2,\cdots,h\}$, $r_id_i=r_1d_1+\sum_{k=2}^i(m_k-m_{k-1})d_k=\sum_{k=1}^{i-1}(d_k-d_{k+1})m_k+m_id_i$.
\end{defi}

\begin{rem}\label{belongsornot}
Let the notations be as in Definition \ref{mrd} and let $v$ be a non zero vector  in $\mathbb{Z}^{e}$. Let $\tilde{D}$ be the $gcd$ of the $(e,e)$ minors of the matrix $(nI_{e},m_{1}^{T},\cdots,m_{i}^{T},v^{T})$, then  $v\in (n\mathbb{Z})^{e}+\sum_{j=1}^{i}m_{j}\mathbb{Z}$ if and only if $D_{i+1}=\tilde{D}$. More generally, $\frac{D_{i+1}}{\tilde{D}}v\in (n\mathbb{Z})^{e}+\sum_{j=1}^{i}m_{j}\mathbb{Z}$ and if $D_{i+1}>\tilde{D}$ then for all $1\leq k< \frac{D_{i+1}}{\tilde{D}}, kv\notin (n\mathbb{Z})^{e}+\sum_{j=1}^{i}m_{j}\mathbb{Z}$.
\end{rem}


\begin{prop}\label{fimi} For all $i=1,\cdots,h-1$ let $H_{i}=L(M={\underline{x}^{m\over n}}, \frac{m}{n}\in{\rm Supp}(y), m<m_{i+1})$. Then we have\\
$(i)$ $F_{i}=H_{i}$ and $m_{i}$ does not belong to $F_{i-1}$ \\ 
$(ii)$ $[F_{i}:F_{i-1}]$, the degree of extension of $F_{i}$ over $F_{i-1}$, is equal to $e_{i}$.   \end{prop}

\begin{proof}{.} $(i)$ Since $m_{j}<m_{i+1}$ for all $j=1,\cdots,i$, then $m_{1},\cdots,m_{i}\in H_{i}$, and so $F_{i}\subseteq H_{i}$. In order to prove that $H_{i}\subseteq F_{i}$, consider a monomial $M$ of $y$  such that $M<M_{i+1}$. For each $\theta\in Aut(L_{n}/F_{i})$, $\theta$ is an $L$ automorphism of $L_{n}$ and $\theta(M_{j})=M_{j}$ for all $j<i+1$. Hence $LM(\theta(y)-y)\geq M_{i+1}$, and so $\theta(M)=M$ for all $M<M_{i+1}$, hence $M\in F_{i}$. Finally we get that $H_{i}=F_{i}$. Now to prove that $m_{i}\notin F_{i-1}$, let $\theta\in Aut(L_{n}/ L)$ such that $\theta(y)-y=cM_{i}+\varepsilon$ with $O(\varepsilon)>m_{i}$ and $c$ a non zero constant (such a $\theta$ obviously exists since $M_{i}$ is a characteristic monomial of $f$), then $\theta(M_{j})=M_{j}$ for all $j=1,\cdots,i-1$ and $\theta(M_{i})\neq M_{i}$, and so $\theta\in Aut(L_{n}/F_{i-1})$ with  $\theta(M_{i})\neq M_{i}$, hence $M_{i}$ does not belong to $F_{i-1}$.\\
$(ii)$ Since $M_{i}\notin F_{i-1}$, then $m_{i}\notin G_{i-1}$, and so $D_{i}>D_{i+1}$. Moreover $e_{i}m_{i}\in G_{i-1}$ and  for all $ 0< \alpha < e_{i}$ we have  $\alpha m_{i}\notin G_{i-1}$. Now let $g=y^{l}+a_{1}y^{l-1}+...+a_{l}$ be the minimal polynomial of $M_{i}$ over $F_{i-1}$ and suppose that $l<e_{i}$. Since $g(M_{i})=0$, then there exists some $k\in\{0,\cdots,l-1\}$ such that $\underline{x}^{l\frac{m_{i}}{n}}=\underline{x}^{\frac{\alpha}{n}}\underline{x}^{\frac{km_{i}}{n}}$ for some $\alpha\in G_{i-1}$, and so $(l-k)m_{i}=\alpha\in G_{i-1}$ with $0<l-k<e_{i}$ which is a contradiction. Hence $l\geq e_i$. But $g$ divides $Y^{e_{i}}-x^{e_{i}\cdot \frac{m_{i}}{n}}$. Hence $g=Y^{e_{i}}-x^{e_{i}\cdot \frac{m_{i}}{n}}$, and consequently $[F_{i}:F_{i-1}]=e_{i}$.$\blacksquare$

\end{proof}
\begin{prop}\label{er}
Let the notations be as above. For all $i\in\{1,\cdots,h\}$ we have $e_{i}r_{i}\in(n\mathbb{Z})^{e}+\sum_{j=1}^{i-1}r_{j}\mathbb{Z}$. Moreover, $\alpha r_{i}\notin (n\mathbb{Z})^{e}+\sum_{j=1}^{i-1}r_{j}\mathbb{Z}$ for all $1\leq \alpha<e_{i}$.
\end{prop}
\begin{proof}{.} We can easily prove that $
    r_{i}
    =m_{i}+\sum_{j=1}^{i-1}(e_j-1)r_j
$ for all  $i\in \{2,\cdots,h\}$, hence  each of the sequences $(m_{k})_{1\leq k\leq h}$ and $(r_{k})_{1\leq k\leq h}$
 can be obtained from the other and $(n\mathbb{Z})^{e}+\sum_{j=1}^{i}r_{j}\mathbb{Z}=(n\mathbb{Z})^{e}+\sum_{j=1}^{i}m_{j}\mathbb{Z}$ for all $i\in\{1,\cdots,h\}$. In particular, for all $\alpha\in{\mathbb N}$, 
 $\alpha r_{i}\in(n\mathbb{Z})^{e}+\sum_{j=1}^{i-1}r_{j}\mathbb{Z}$ if and only if $\alpha m_{i}\in(n\mathbb{Z})^{e}+\sum_{j=1}^{i-1}m_{j}\mathbb{Z}$. Let $i\in\{1,\cdots,h\}$. By Remark \ref{belongsornot},  $e_{i}m_{i}=\frac{D_{i}}{D_{i+1}}m_{i}\in (n\mathbb{Z})^{e}+\sum_{j=1}^{i-1}m_{j}\mathbb{Z}$ and $\alpha m_{i}\notin (n\mathbb{Z})^{e}+\sum_{j=1}^{i-1}m_{j}\mathbb{Z}$ for all $1\leq \alpha< e_{i}$. Hence  $e_{i}r_{i}\in(n\mathbb{Z})^{e}+\sum_{j=1}^{i-1}r_{j}\mathbb{Z}$ and $\alpha r_{i}\notin(n\mathbb{Z})^{e}+\sum_{j=1}^{i-1}r_{j}\mathbb{Z}$ for all $1\leq \alpha<e_{i}. \blacksquare$
\end{proof}

\begin{rem}\label{ehdh}
Since $[L(y):L]=n$, then it follows from proposition $\ref{fimi}$ that $[L(y):L]=e_{1}\cdots e_{h}=\frac{D_{1}}{D_{h+1}}$. But $[L(y):L]=n$ and $D_{1}=n^{e}$, hence $D_{h+1}=n^{e-1}$. It follows that $d_{1}=n$ and $d_{h+1}=1$.
\end{rem}

\noindent For all $i\in \lbrace 1,\cdots,h\rbrace$, define the following sets
$Q(i)=\{\theta\in Aut(L_{n}/L) \mid  nO(y-\theta(y))<m_i\},
R(i)=\{\theta\in Aut(L_{n}/L) \mid nO(y-\theta(y))\geqslant m_i\}$ and $
S(i)=\{\theta\in Aut(L_{n}/L) \mid nO(y-\theta(y))=m_i\}$. With these notations we have the following:

\begin{prop}\label{sidi}  $\# R(i)=D_i$ and $\# S(i)=D_{i}-D_{i+1}$, where $\#$ stand for  the cardinality.
 \end{prop}

\begin{proof}{.} We have $\theta\in R(i) \Leftrightarrow \theta(M_j)=M_j$ for all $j<i \Leftrightarrow \theta\in Aut(L_{n}/L(M_{1},\cdots,M_{i-1}))$,
hence $\# R(i)=\# Aut(L_{n}/L(M_{1},\cdots,M_{i-1}))=[L_{n}:L(M_{1},\cdots,M_{i-1})]=[L_{n}:F_{i-1}]$. By proposition $\ref{fimi}$ we have $
[F_{i-1}:L]=[F_{i-1}:F_{i-2}]\cdots [F_{1}:L]=e_{i-1}\cdots e_{1}
=\frac{D_{1}}{D_{i}}=\frac{n^{e}}{D_{i}}
$. But  $[L_{n}:L]=[L_{n}:F_{i-1}] [F_{i-1}:L]=n^{e}$, then $[L_{n}:F_{i-1}]=D_{i}$, and so $\# R(i)=D_{i}$. Now $R(i+1)\subset R(i)$ and  $\theta\in S(i)$ if and only if $nO(y-\theta(y))= m_i$ if and only if $\theta\in R(i)$ and $\theta\notin R(i+1)$, hence $\#S(i)=D_{i}-D_{i+1}$. $\blacksquare$
\end{proof}

\noindent Similarly to Proposition \ref{sidi} we get  the following: for all $i\in\lbrace 1,\cdots,h\rbrace$, let $\tilde{R}(i)=\lbrace y_k \mid nO(y-y_k)\geq m_i \rbrace$ and $\tilde{S}(i)=\lbrace y_k \mid nO(y-y_k)=m_i\rbrace$. We have: 

\begin{prop}\label {sidi1} $\# \tilde{R}(i)=d_i$ and $\# \tilde{S}(i)=d_{i}-d_{i+1}$.
\end{prop}
\subsection{Pseudo roots, semigroup, and approximate roots of a free polynomial}
Let the notations be as above. For all $i\in\{1,\cdots,h\}$ we will define a specific free polynomial $G_{i}$, called the $i^{th}$ pseudo root of $f$ such that $O(G_{i}(\underline{x},y(\underline{x})))=\frac{r_{i}}{n}$. Also we will define the semigroup $\Gamma(f)$ of $f$ and we will construct a system of generators of $\Gamma(f)$. Finally we will prove that  $O(f,{\rm App}(f,d_i))=r_i$   for all $i\in\lbrace 1,\cdots,h\rbrace$ (see Definition 6 below). Let $y(x)=\sum c_{p}\underline{x}^{\frac{p}{n}}$ be a root of $f$ and let $\frac{\underline{m}}{n} \in{\rm Supp}(y)$. We set $y_{<m}=\sum_{p<m}c_{p}\underline{x}^{\frac{p}{n}}$ and we call $y_{<m}$ the $m$-truncation of $y$.

\begin{defi} Let the notations be as above. Given $g\in{\mathbb K}_C[\![\underline{x}]\!][y], f\not |g$, we set $O(f,g)=\sum_{i=1}^nO(g(\underline{x},y_i))$\\$=nO(g(\underline{x},y(\underline{x})))$. Clearly $O(f,g_1g_2)=O(f,g_1)+O(f,g_2)$. It follows that $\Gamma(f)=\lbrace O(f,g)|g\in {\mathbb K}_C[\![\underline{x}]\!][y]\setminus(f)\rbrace$ is a semigroup. We call it the semigroup associated with $f$.
\end{defi}

\noindent In the following we will  prove that $(r_0^1,\cdots,r_0^e,r_1,\cdots,r_h)$ is a system of generators of $\Gamma(f)$. This will be done by using a set of polynomials called pseudo roots of $f$.

\begin{defi} For all $i\in\lbrace 1,\cdots,h\rbrace$, we define the $i^{th}$ pseudo root of $f$ to be the minimal polynomial of $y_{<m_{i}}$  over $L$. We denote it by $G_i$.
\end{defi}

\noindent In the following we shall study the properties of $G_i$. In particular we shall prove that $O(f,G_i)=r_i$.
\begin{prop}
Let the notations be as above. For all $i=1,\cdots,h, {\rm\ deg}_{y}(G_{i})=\frac{n^{e}}{D_{i}}=\frac{n}{d_{i}}$.
\end{prop}
\begin{proof}{.} By proposition $\ref{fimi}$ we have $L(y_{<m_{i}})=L(M_{1},..,M_{i-1})$. In particular $deg_{y}(G_{i})=[L(y_{<m_{i}}):L]=[L(M_{1},\cdots,M_{i-1}):L]=\frac{n^{e}}{D_{i}}=\frac{n}{d_i}.\blacksquare$
\end{proof}

\begin{prop}\label{Gicharacteristic} The polynomial  $G_i$ is free, and its characteristic exponents are $\frac{m_{1}}{d_{i}},\cdots,\frac{m_{i-1}}{d_{i}}$. 
\end{prop}
\begin{proof}{.} The polynomial $G_i$ is free from the definition. We shall prove that $y_{<m_{i}}\in\mathbb{K}_{C}[\![\underline{x}^{\frac{1}{\frac{n}{d_{i}}}}]\!]$. Let $\underline{x}^{\frac{\lambda}{n}}$ be a monomial of $y_{<m_{i}}$, then $\lambda\in(n\mathbb{Z})^{e}+\sum_{j=1}^{i-1}m_{j}\mathbb{Z}$. Let $D$ be the $gcd$ of the minors of the matrix $(m_{0}^{1},\cdots,m_{0}^{e},m_{1},\cdots,m_{i-1},\lambda)$, then $D=D_{i}$. For all $l\in\{1,\cdots,e\}$ the matrix $A_{l}=(m_{0}^{1},\cdots,m_{0}^{l-1},\lambda,m_{0}^{l+1},\cdots,m_{0}^e)$  is one of the minors of the matrix $(m_{0}^{1},\cdots,m_{0}^{e},m_{1},\cdots,m_{i-1})$, then $D_{i}$ divides $Det(A_{l})$. Write $\lambda=(\lambda_{1},\cdots,\lambda_{e})$, then obviously $Det(A_{l})=n^{e-1}\lambda_{l}$, and so $D_{i}$ divides $n^{e-1}\lambda_{l}$ for all $l\in\{1,\cdots,e\}$. It follows that  $\frac{n^{e-1}\lambda}{D_{i}}=\frac{\lambda}{d_{i}}\in\mathbb{Z}^{e}$. Moreover, since $\lambda\in C$, and $\frac{1}{d_{i}}\geq 0$, then $\frac{\lambda}{d_{i}}\in C$. Hence $\underline{x}^{\frac{\lambda}{n}}=\underline{x}^{\frac{\lambda'}{\frac{n}{d_{i}}}}$ where $\lambda'=\frac{\lambda}{d_{i}}$, and so $\underline{x}^{\frac{\lambda}{n}}\in\mathbb{K}_{C}[\![\underline{x}^{\frac{1}{\frac{n}{d_{i}}}}]\!]$.\\
Let $\theta(y_{<m_{i}})$ be a conjugate of $y_{<m_{i}}$, then obviously $LM(\theta(y_{<m_{i}})-y_{<m_{i}})=\underline{x}^{\frac{m_{j}}{n}}$ for some $j\in\{1,\cdots,i-1\}$. But $\frac{m_{j}}{n}=\frac{\frac{m_{j}}{d_{i}}}{\frac{n}{d_{i}}}$, hence the set of  characteristic exponents of $G_i$ is $\{\frac{m_{1}}{d_{i}},\cdots,\frac{m_{i-1}}{d_{i}}\}$.$\blacksquare$

\end{proof}

\begin{prop}\label{qqofi} Let the notations be as above. For all $i\in\lbrace 1,\cdots,h\rbrace$, we have $O(f(\underline{x},y_{<m_i}(\underline{x})))=\dfrac{r_id_i}{n}$.

 \end{prop}
\begin{proof}{.} We have $f(\underline{x},y_{<m_i})=\prod_{k=1}^n(y_{<m_i}-y_k)$ with the assumption that $y=y_1$. Clearly $O(y_{<m_i}-y_k)=O(y_1-y_k)$ if $O(y_1-y_k)<\frac{m_i}{n}$ and $\frac{m_i}{n}$ otherwise. It follows from Proposition \ref{sidi1} that $O(\prod_{k=1}^n(y_{<m_i}-y_k)=\frac{1}{n}(\sum_{k=1}^{i-1}(d_k-d_{k+1})m_k+d_im_i)$, which is equal to $\dfrac{r_id_i}{n}$ by Definition \ref{mrd}.$\blacksquare$
\end{proof}

\noindent Let $g=y^m+b_1(\underline{x})y^{m-1}+\cdots+b_m(\underline{x})$ be a free polynomial of ${\mathbb K}_C[\![\underline{x}]\!][y]$ and let $ z_1,\cdots,z_m$ be the set of roots of $g$ in $\mathbb{K}[\![\underline{x}^{\frac{1}{m}}]\!]$. We set $O(f,g)=\sum_{i=1}^nO(g(\underline{x}, y_i(\underline{x})))$. Clearly $O(f,g)=\sum_{j=1}^mO(f(\underline{x}, z_j(\underline{x})))=O(g,f)=O({\rm Res}_y(f,g))$, where Res stand for the $y$-resultant of $f,g$. As a corollary of Proposition \ref{qqofi} we get the following:

\begin{cor}\label{ordps} With the notations above, we  have $O(f,G_i)=r_i$

\end{cor}

\begin{proof}{.} In fact, $O(f,G_i)=O(G_i,f)=\frac{n}{d_i}O(f(\underline{x},y_{<m_i}))=r_i$.$\blacksquare$

\end{proof}

\noindent As a corollary we get the following:

\begin{prop}\label{GiGj}
 Let 
 $\{G_{1},\cdots,G_{h}\}$ be the set of pseudo roots of $f$. Let $i\in\{1,\cdots,h\}$, then we have $O(G_{i},G_{j})=\frac{r_{j}}{d_{i}}$ for all $j\in\{1,\cdots,i-1\}$.
\end{prop}
\begin{proof}{.} This is an immediate consequence of Corollary \ref{ordps} because $\{G_1,\cdots,G_{i-1}\}$ is the set of pseudo-approximate roots of $G_i$ and the $\underline{r}$ sequence of $G_i$ is given by $\frac{r_0^1}{d_i},\cdots,\frac{r_0^e}{d_i},\frac{r_1}{d_i},\cdots,\frac{r_{i-1}}{d_i}$. $\blacksquare$
\end{proof}

\noindent Next we shall  prove that $(r_0^1,\cdots,r_0^e,r_1,\cdots,r_h)$ is a system of generators of $\Gamma(f)$. We shall need the following result:

\begin{lemma} \label{uniqueness}Let the notations be as above and let $\underline{\alpha}=(\alpha_0^1,\cdots,\alpha_0^e,\alpha_1,\cdots,r_h), \underline{\beta}=(\beta_0^1,\cdots,\beta_0^e,\beta_1,\cdots,\beta_h)$ be two elements of ${\mathbb Z}^e\times {\mathbb N}^h$ such that $0\leq \alpha_i,\beta_i<e_i$ for all $i\in\lbrace 1,\cdots,h\rbrace$. If $a=\sum_{i=1}^e\alpha_0^ir_0^i+\sum_{j=1}^h\alpha_jr_j=\sum_{i=1}^e\beta_0^ir_0^i+\sum_{j=1}^h\beta_jr_j$  then $\underline{\alpha}=\underline{\beta}$. 
\end{lemma}

\begin{proof}{.} Suppose that $\underline{\alpha}\not=\underline{\beta}$ and let $k$ be the smallest integer $\geq 1$ such that $\alpha_i=\beta_i$ for all $i\geq k+1$. Suppose that $\alpha_k>\beta_k$. We have $(\alpha_k-\beta_k)r_k=\sum_{i=1}^e(\beta_0^i-\alpha_0^i)r_0^i+\sum_{j=1}^{k-1}(\beta_j-\alpha_j)r_j$. This contradicts Proposition \ref{er}.$\blacksquare$ 
\end{proof}

\begin{lemma}\label{pseudoexpansion}
Let  $g\in\mathbb{K}_{C}[\![\underline{x}]\!][y]$ and suppose that $f\not|g$. There exists a unique $\underline{\theta}=(\theta_0^1,\cdots,\theta_0^e,\theta_1,\cdots,\theta_h)\in {\mathbb Z}^e\times {\mathbb N}^h$ such that $0\leq \theta_j<e_j$ for all $j\in\lbrace 1,\cdots,h\rbrace$ and   $O(f,g)=\sum_{i=1}^e\theta_0^ir_0^i+\sum_{j=1}^h\theta_jr_j$. In particular $\Gamma(f)$ is generated by $r_0^1,\cdots,r_0^e,r_1,\cdots,r_h$.
 
\end{lemma}
 
\begin{proof}{.} \label{Exp} Let $g=\sum_{\underline{\theta}}c_{\underline{\theta}}(\underline{x})G_{1}^{\theta_1}\cdots G_{h}^{\theta_{h}}f^{\theta_{h+1}}$ be the expansion of $g$ with respect to $(G_1,\cdots,G_h,f)$ and recall that for all $\underline{\theta}$, if $c_{\underline{\theta}}\not=0$ then  $\underline{\theta}=(\theta_{1},\cdots,\theta_{h+1})\in \{(\beta_{1},\cdots,\beta_{h+1})\in{\mathbb N}^{h+1}, \ \ 0\leq \beta_{j}<e_{j}\ \forall j=1,\cdots,h\}$. By abuse of notations we shall call a monomial a term of the form $M=c_{\underline{\theta}}(\underline{x})G_{1}^{\theta_1}\cdots G_{h}^{\theta_{h}}f^{\theta_{h+1}}$ The hypothesis implies that there exists at least one $\underline{\theta}$ such that $c_{\underline{\theta}}\not= 0$ and $\theta_{h+1}=0$. Let $M=c_{\underline{\theta}}(\underline{x})G_{1}^{\theta_1}\cdots G_{h}^{\theta_{h}},$ $ N=c_{\underline{\theta}'}(\underline{x})G_{1}^{\theta'_1}\cdots G_{h}^{\theta_{h}'}$ be two distinct monomials of $g$. It follows from Lemma \ref{uniqueness} that $O(f,M)\not= O(f,N)$. Hence there exists a unique monomial $\tilde{M}$ of $g$ such that $O(f,g)=O(f,\tilde{M})$. This proves our assertion. $\blacksquare$
\end{proof} 
\begin{rem}\label{nonzero} In the Lemma above, if deg$_yg<\frac{n}{d_i}$ for some $i\in\lbrace 1,\cdots,h\rbrace$, then 
$O(f,g)\in (n{\mathbb Z})^e+\sum_{k=1}^{i-1}r_k{\mathbb N}$. Moreover, $O(f,g)=d_{i}O(G_{i},g)$.  In fact, in this case, any monomial $M$ of the expansion of $g$ with respect to $(G_1,\cdots,G_h,f)$ is a monomial in $G_1,\cdots,G_{i-1}$. Hence this expansion coincides with that of $g$ with respect to $(G_1,\cdots,G_{i-1},G_i)$. If $M$ is the unique monomial such that $O(f,g)=O(f,M)$  then $M$ is the unique monomial such that $O(G_i,g)=O(G_i,M)$. But $O(f,M)=d_iO(G_i,M)$. This proves our assertion.

\end{rem}
 

\noindent The next Proposition shows that we can calculate a system of generators of $\Gamma(f)$ only with the set of approximate roots of $f$. It uses Lemma \ref{pseudoexpansion} and Remark \ref{nonzero}, and the proof is similar to the proof of similar results in other situations (see \cite{Abhyankar5}, \cite{A.Assi1}, or \cite{GP}).

\begin{prop}\label{Gdeformation} 
For all $i\in\lbrace 1,\cdots,h\rbrace$, let $g_i={\rm App}(f,d_i)$.  We have $O(f,g_{i})=r_{i}$.
\end{prop}

\begin{proof}{.}
Let $i=h$ and consider the $G_{h}$-adic expansion of $f$, 
$f=G_{h}^{d_{h}}+C_{1}(\underline{x},y)G_{h}^{d_{h}-1}+\cdots+C_{d_{h}}(\underline{x},y)=\sum_{k=0}^{d_{h}}C_{k}(\underline{x},y) G_{h}^{d_{h}-k}$
where $C_0=1$ and $C_{k}(\underline{x},y)\in\mathbb{K}_{C}[\![\underline{x}]\!][y]$ with $deg_{y}(C_{k}(\underline{x},y))<\frac{n}{d_{h}}$ for all $k=1,\cdots,d_{h}$.  Consider the Tschirnhausen transform of $G_{h}$ with respect to $f$ given by $\tau_{f}(G_{h})=G_{h}+d_{h}^{-1} C_{1}(\underline{x},y).$
We have   $O(f,G_{h})=r_{h}$, hence we need to prove that $O(f,C_{1})>r_{h}$.\\ 
Let $k\in\{0,\cdots,d_{h}-1\}$.  For all $\alpha\neq k$, we have $O(f,C_{\alpha}G_{h}^{d_{h}-\alpha})\neq O(f,C_{k}G_{h}^{d_h-k})$. In fact, suppose that  $O(f,C_{\alpha}G_{h}^{d_{h}-\alpha})= O(f,C_{k}G_{h}^{d_h-k})$, that is $O(f,C_{\alpha})+(d_{h}-\alpha)r_{h}=O(f,C_{k})+(d_{h}-k)r_{h}$. Suppose that $\alpha>k$, then $(\alpha-k)r_{h}=O(f,C_{\alpha})-O(f,C_{k})$. But deg$_{y}(C_{\alpha}),{\rm deg}_{y}(C_{k})<\frac{n}{d_{h}}$, then by Remark \ref{nonzero}, $O(f,C_{\alpha}),O(f,C_{k})\in(n\mathbb{Z})^{e}+r_{1}\mathbb{N}+\cdots+r_{h-1}\mathbb{N}$, and so $(\alpha-k)r_{h}\in (n\mathbb{Z})^{e}+r_{1}\mathbb{N}+\cdots+r_{h-1}\mathbb{N}$, with $0<\alpha-k<d_{h}=e_h$. This contradicts Proposition \ref{er}. Now a similar argument shows that $O(f,C_{k}G_{h}^{d_{h}-k})=O(f,C_{k})+(d_{h}-k)r_{h}\neq O(f,C_{d_{h}})$. As $f(\underline{x},y(\underline{x}))=0$,  we get that $O(f,C_{d_h})=O(f,G_h^{d_h})=r_hd_h<O(C_kG_k^{d_h-k})$, hence $O(f,C_k)>kr_h$. This is true for $k=1$, consequently $O(f,C_1)>r_h$, and $O(f,\tau_f(G_h))=r_h$. Repeating this process, we get that $O(f,\tau^{l}_{f}(G_{h}))=r_{h}$ for all $l\geq 1$. But $g_{h}={\rm App}(f,d_h)=\tau^{l_0}_{f}(G_{h})$ for some $l_0$. Hence $O(f,g_{h})=r_{h}$.\\
Now suppose that $O(f,g_{k})=r_{k}$ for all $k>i$, and let us prove that $O(f,g_{i})=r_{i}$. Note that  $g_{i}={\rm App}(g_{i+1},e_i)$.  Let  
\begin{equation}\label{secondgi}
g_{i+1}=G_{i}^{e_{i}}+\beta_{1}(\underline{x},y)G_{i}^{e_{i}-1}+\cdots+\beta_{e_{i}}(\underline{x},y)
\end{equation}
be the $G_{i}$-adic expansion of $g_{i+1}$ and consider $O(f,g_{i+1})$. For all $k\in\lbrace 1,\cdots,e_i\rbrace$,   $O(f,\beta_kG_{k}^{e_i-k})=O(f,\beta_k)+(e_i-k)r_i$. But $O(f,\beta_k)\in (n{\mathbb Z})^e+\sum_{j=1}^{i-1}r_j{\mathbb N}$ because deg$_y\beta_k<\frac{n}{d_i}$, and $r_{i+1}\notin (n{\mathbb Z})^e+\sum_{j=1}^{i}r_j{\mathbb N}$. Now a similar argument as above shows that $r_ie_i=O(f,G_i^{e_i})=O(f,\beta_{e_i})<O(\beta_1G_i^{e_i-1})$. Hence $O(f,\beta_1)>r_i$. In particular 
$$O(f,\tau_{g_{i+1}}(G_{i}))=O(f,G_{i}+\frac{1}{e_{i}}\beta_{1})=r_{i}$$
Applying the same process to $f$ and $\tau_{g_{i+1}}(G_{i})$ instead of $f$ and $G_{i}$. We get that $O(f,\tau_{g_{i+1}}^{2}(G_{i}))=r_{i}$. But $g_i=\tau_{g_{i+1}}^{{s}_{i}}(G_{i}))$ for some $s_i\in {\mathbb N}$, hence $O(f,g_{i})=O(f,\tau_{g_{i+1}}^{e_{i}}(G_{i}))=r_{i}$. This proves our assertion. $\blacksquare$ 
\end{proof}


\begin{rem}\label{free} Let the notations be as above. The $d$-sequence  $\lbrace d_i\rbrace_{1\leq i\leq h+1}$ introduced in Definition 5 satisfies $d_1=n>d_2>\cdots >d_{h+1}=1$. Moreover,  by Proposition \ref{er}, for all   $i\in \lbrace 1,\cdots,h\rbrace$, we have $e_{i}r_{i}\in(n\mathbb{Z})^{e}+\sum_{j=1}^{i-1}r_{j}\mathbb{Z}$. Following the notations of \cite{A.Assi2}, the semigroup $\Gamma(f)$ is a free affine semigroup with respect to the arrangement $(r_0^1,\cdots,r_0^e,r_1,\cdots,r_h)$ (this notion has been introduced first for numerical semigroups, i.e. monoids of ${\mathbb N}$ with finite complement in ${\mathbb N}$). Referring to free affine semigroups, we have chosen to use here the notion of free polynomials.
\end{rem}
\section{Solutions of formal power series}
Let $f(\underline{x},y)=y^{n}+a_{1}(\underline{x})y^{n-1}+\cdots+a_{n-1}(\underline{x})y+a_{n}(\underline{x})$ be a polynomial of degree $n$ in $\mathbb{K}[\![\underline{x}]\!][y]$. In this section we shall apply the results of Section 3 to $f$ seeing as a polynomial in $y$ whose coefficients are in $\mathbb{K}_{C}[\![\underline{x}]\!]$ for a specific line free cone $C$. We first connect, modulo a preparation result, the polynomial $f$ to a quasi-ordinary polynomial, which is irreducible if and only if $f$ is irreducible in $\mathbb{K}_{C}[\![\underline{x}]\!][y]$, and in this case, it is free. Hence the set of roots of the quasi-ordinary polynomials are connected with the set of roots of $f$ in $\mathbb{K}_{C}[\![\underline{x}^{\frac{1}{n}}]\!]$. We start with the following preparation result.


\noindent Let $\Delta(\underline{x})$ be the discriminant of $f$ in $y$, and write $\Delta(\underline{x})=\sum_{p\in \mathbb{N}^{e}}c_{p}\underline{x}^{p} =\sum_{d\geq 0}u_{d}(\underline{x})$ where for all $d\geq 0$,
$u_{d}$  is the homogeneous component of degree $d$ of $\Delta$. Let $a=inf\{d,\ u_{d}\neq 0\}$. If  $a=0$, then $f$ is a quasi-ordinary polynomial. Suppose that $a > 0$. In the next remark we will show how to prepare our polynomial so that the smallest homogeneous component $u_a$ of $\Delta$ contains a monomial in $x_1$.


\begin{rem}\label{psiisomorphism}({\bf Preparation}) Consider the mapping $\xi:\mathbb{K}[\![\underline{x}]\!]\mapsto \mathbb{K}[\![\underline{X}]\!]$,
defined by $\xi(x_{1})=X_{1}$ and $\xi(x_{i})=X_{i}+t_iX_{1}$ for all $i\in\{2,\cdots,e\}$, where $t_2,\cdots, t_n$ are parameters. Let 

$$
\psi:\ \mathbb{K}[\![\underline{x}]\!][y]\ \mapsto\ \mathbb{K}[\![\underline{X}]\!][y]
$$

\noindent  be the map defined as follows: if $H=h_{0}(\underline{x})y^{m}+\cdots +h_{m-1}(\underline{x})y+h_{m}(\underline{x})\in \mathbb{K}[\![\underline{x}]\!][y]$  then $\psi(H)=\xi(h_{0}(\underline{x}))y^{m}+\cdots +\xi(h_{m-1}(\underline{x}))y+\xi(h_{m}(\underline{x}))$. Then we easily prove that $\psi$ is an isomorphism. If $\Delta'$ is the discriminant of $\psi(f)$ and if  $v_d(\underline{X})=u_{d}(X_{1},X_{2}+t_2X_{1},\cdots,X_{e}+t_eX_{1})$ then  $\Delta'=\sum_{d\geq a} v_d$. But
$v_{d}(\underline{X})=\varepsilon_{d}(t_2,\cdots, t_e)X_{1}^{d}+v'_d$,  
where $v'_{d}$ is a homogeneous polynomial of degree $d$, and $\varepsilon_{d}(t_2,\cdots,t_e)$ is a polynomial in $t_2,\cdots,t_e$. We claim that $\varepsilon_{a}(t_2,\cdots,t_e)$ is a nonzero polynomial, hence we can choose $t_2,\cdots, t_e \in\mathbb{K}$ such that $\varepsilon_{a}(t_2,\cdots,t_e)\neq 0$.  In fact, let 

$$
u_a=\sum_{k=1}^m c_kx_1^{a_1^k}\cdots  x_e^{a_e^k}
$$

\noindent  with $a_1^k+\cdots +a_2^k=a, c_k\neq 0$ for all $k\in\lbrace 1,\cdots,m\rbrace$, and  $(a_1^k,\cdots,a_e^k) \neq  (a_1^j,\cdots,a_e^j)$ for all $k\neq j$. In particular  $(a_2^k,\cdots,a_e^k) \neq  (a_2^j,\cdots,a_e^j)$ for all $k\neq j$. We have: 

$$
u_a(X_1,X_2+t_2X_1,\cdots,X_e+t_eX_1)=\sum_{k=1}^m c_kX_1^{a_1^k}(X_2+t_2X_1)^{a_2^k}\cdots (X_e+t_eX_1)^{a_e^k}
$$

$$
= \sum_{k=1}^mc_kX_1^{a_1^k}(t_2X_1)^{a_2^k}\cdots (t_eX_1)^{a_e^k}+v'_a=\sum_{k=1}^mc_kt_2^{a_2^k}\cdots t_e^{a_e^k} X_1^{a_1^k} X_1^{a_2^k}\cdots X_1^{a_e^k}+v'_a
$$
$$
= \sum_{k=1}^m c_k t_2^{a_2^k}\cdots t_e^{a_e^k} X_1^{a_1^k+a_2^k+\cdots+a_e^1}+v'_a=(\sum_{k=1}^m c_k t_2^{a_2^k}\cdots t_e^{a_e^k}) X_1^{a}+v'_a
$$
 \noindent where $v_{a}$ is a homogeneous polynomial of degree $a$, such that $v_a(1,0,\cdots,0)=0$. Since ($a_2^k,\cdots,a_e^k) \neq  (a_2^j,\cdots,a_e^j)$ for all $k\neq j$ and $c_k\neq 0$ for all $k\in\lbrace 1,\cdots, m\rbrace$,  then $\varepsilon_{a}(t_2,\cdots, t_e)=\sum_{k=1}^m c_k t_2^{a_2^k}\cdots t_e^{a_e^k}$ is a non zero polynomial. Hence, we can choose $t_2,\cdots,t_e\in\mathbb{K}$ such that  $\varepsilon_{a}(t_1,\cdots,t_e)\neq 0$.
\end{rem}


\noindent In the following we shall say that a polynomial $f$ is prepared if it satisfies the condition of Remark \ref{psiisomorphism}, i.e. its discriminant is of the form  $\Delta=\sum_{d\geq 0}u_{d}$ such that the smallest homogeneous component is of the form $u_{a}=c_{a}x_{1}^{a}+u'_{a}$  with $c_a\neq 0$ and $u'_a\in{\mathbb K}[\underline{x}]$.  The next proposition shows that a prepared polynomial is birationally equivalent to a quasi-ordinary polynomial.
\begin{prop}\label{change}
With the notations above, if $f$ is a prepared polynomial then $F(X_{1},\cdots,X_{e},y)=f(X_{1},X_{2}X_{1},\cdots, X_{e}X_{1},y)$ is a quasi-ordinary polynomial.
\end{prop}

\begin{proof}{.} Let $\Delta$ be the discriminant of $f$. The  discriminant $\Delta_{N}$ of $F$ is $ \Delta_{N}=\Delta(X_{1},X_{2}X_{1},\cdots, X_{e}X_{1})$.
Write $\Delta=\sum_{d\geq a}u_{d}$, where $u_{d}$ is the homogeneous component of degree $d$ of $\Delta$ and $u_a\not=0$, then 
$\Delta_{N}=\sum_{d\geq a}w_{d}(\underline{X})$ with $w_d(\underline{X})=u_d(X_1,X_2X_1,\cdots,X_eX_1)$. For all $d\geq a$, we have  $$w_d(\underline{X})=X_1^du_d(1,X_2,\cdots,X_e)=X_{1}^{d}(c_{d}+\varepsilon_{d}(X_{1},\cdots,X_{e}))=X_{1}^{a}X_{1}^{d-a}(c_{d}+\varepsilon_{d}(X_{1},\cdots,X_{e}))$$
\noindent where $c_d\in{\mathbb K}$ and $\varepsilon_{d}(0,\cdots,0)=0$. Since $f$ is prepared, then $c_a \neq 0$, hence 
$\Delta_{N}=X_{1}^{a}(c_a+\varepsilon(\underline{X}))$  and $\varepsilon(\underline{X})$ is a non unit in $\mathbb{K}[\![\underline{X}]\!]$. So $F$ is a quasi-ordinary polynomial. $\blacksquare$

\end{proof}

\noindent We will now introduce the following line free cone.

\begin{prop}\label{conecone}
The set $C=\{(c_{1},\cdots,c_{e})\in \mathbb{R}^{e},c_{1}\geq -(c_{2}+\cdots+c_{e}),\ c_{i}\geq 0\ \forall\  2\leq i\leq e\}$ is a line free convex cone.
\end{prop}
\begin{proof}{.} Let $c=(c_{1},\cdots,c_{e})\in C$ and $\lambda \geq 0$, then obviously $\lambda c\in C$, hence $C$ is a cone. Moreover, if $c=(c_{1},\cdots,c_{e}),c'=(c_{1}',\cdots,c_{e}')\in C$, then $c+c'\in C$, and so $C$ is a convex cone.  Let $c=(c_{1},\cdots,c_{e})\in C$ such that $c\neq \underline{0}$, and let us prove that $-c=(-c_{1},\cdots,-c_{e})\notin C$. We have $c_{i}\geq 0$ for all $i\in \{2,\cdots,e\}$. If $c_{i}>0$ for some $i\in \{2,\cdots,e\}$, then obviously $-c=(-c_{1},\cdots,-c_{e})\notin C$. If $c_{i}=0$ for all $i\in\{2,\cdots,e\}$, then $c_{1}\geq -(c_{2}+\cdots +c_{e})=0$, but $c\neq \underline{0}$, then $c_{1}>0$, and so $-c=(-c_{1},0,\cdots,0)\notin C$. Hence $C$ is a line free cone.$\blacksquare$
\end{proof}
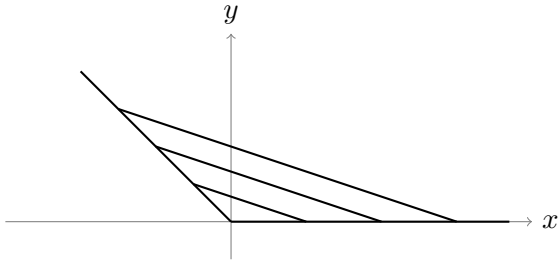
\begin{tikzpicture}
    \draw [thin, gray, ->] (0,-0.5) -- (0,2.5)  
    node [above, black] {$y$};  
    \draw [thin, gray, ->] (-3,0) -- (4,0)   
    node [right, black] {$x$}; 
    
    \draw [draw=black, thick] (0,0) -- (-2,2); 
     \draw [draw=black, thick] (-1,1) -- (2,0);
     \draw [draw=black, thick] (-1.5,1.5) --(3,0);
     \draw [draw=black, thick] (-0.5,0.5) --(1,0);
   \draw [draw=black, thick] (0,0) --(3.7,0);  
\end{tikzpicture}

\noindent Along this Section,  $C$ will denote the cone defined in  proposition $\ref{conecone}$.

\begin{lemma}\label{support}
Let $Y(\underline{X})$ be an element of    $\mathbb{K}[\![\underline{X}]\!]$, and let $y(\underline{x})=Y(x_{1},x_{2}x_{1}^{-1},\cdots, x_{e}x_{1}^{-1})$. We have $y(\underline{x})\in \mathbb{K}_{C}[\![\underline{x}]\!]$.
\end{lemma}
\begin{proof}{.} Write $Y(\underline{X})=\sum_{\underline{a}}
\gamma_{\underline{a}}\underline{X}^{\underline{a}}$,
 then 
$
y(\underline{x})=\sum_{\underline{a}}\gamma_{\underline{a}}
x_{1}^{a_{1}-(a_{2}+\cdots+a_{e})}x_{2}^{a_{2}}\cdots\  x_{e}^{a_{e}}
$. In particular
Supp$(y)=\{(a_{1}-(a_{2}+\cdots+a_{e}),a_{2},\cdots,a_{e}), \underline{a}\in {\rm Supp}(Y)\}.$  As $a_1\geq 0$, we have $a_{1}-(a_{2}+\cdots+a_{e})\geq -(a_{2}+\cdots+a_{e})$, hence $y(\underline{x})\in \mathbb{K}_{C}[[\underline{x}]]$.$\blacksquare$ 
\end{proof}
\noindent The following proposition characterizes the irreducibility of elements of $\mathbb{K}[\![\underline{x}]\!][y]$ in $\mathbb{K}_{C}[\![\underline{x}]\!][y]$. 

\begin{prop}\label{firreducibleF}
With the notations above,  $f$ is irreducible in $\mathbb{K}_{C}[\![\underline{x}]\!][y]$ if and only if $F(X_{1},\cdots,X_{e},y)=f(X_{1},X_{2}X_{1},\cdots,X_{e}X_{1},y)$ is irreducible in $\mathbb{K}[\![\underline{X}]\!][y]$. 
\end{prop}
\begin{proof}{.}
Suppose that $f$ is irreducible in $\mathbb{K}_{C}[\![\underline{x}]\!][y]$. If $F$ is reducible in $\mathbb{K}[\![\underline{X}]\!][y]$, then there exist monic  polynomials $G,H\in \mathbb{K}[\![\underline{X}]\!][y]$  such that $F=GH$ and $0<deg_{y}(G),deg_{y}(H)<n$. But $f(x_{1},\cdots,x_{e},y)=F(x_{1},x_{2}x_{1}^{-1},\dots ,x_{e}x_{1}^{-1},y)$. Then:
$$f(x_{1},\cdots,x_{e},y)=G(x_{1},x_{2}x_{1}^{-1},\dots, x_{e}x_{1}^{-1},y)H(x_{1},x_{2}x_{1}^{-1},\dots, x_{e}x_{1}^{-1},y).$$
Let  $g(\underline{x},y)=G(x_{1},x_{2}x_{1}^{-1},\dots, x_{e}x_{1}^{-1},y)$ and $h(\underline{x},y)=H(x_{1},x_{2}x_{1}^{-1},\dots, x_{e}x_{1}^{-1},y)$.  Let $m=deg_{y}(G)$ and write $G(\underline{X},y)=y^{m}+a_{1}(\underline{X})y^{m-1}+\cdots+a_{m}(\underline{X})$, where $a_{i}(\underline{X})\in\mathbb{K}[\![\underline{X}]\!]$ for all $i\in\lbrace 1,\cdots,m\rbrace$. We have:
$$g(\underline{x},y)=y^{m}+a_{1}(x_{1},x_{2}x_{1}^{-1},\dots, x_{e}x_{1}^{-1})y^{m-1}+\cdots+a_{m}(x_{1},x_{2}x_{1}^{-1},\dots, x_{e}x_{1}^{-1})$$
Since $a_{i}(\underline{X})\in\mathbb{K}[\![\underline{X}]\!]$ for all $i=1,\cdots,m$, then by Lemma $\ref{support}$ we get that $a_{i}(x_{1},x_{2}x_{1}^{-1},\dots, x_{e}x_{1}^{-1})\in\mathbb{K}_C[\![\underline{x}]\!]$ for all $i=1,\cdots,m$. It follows that $g\in\mathbb{K}_{C}[\![\underline{x}]\!][y]$. Similarly we can prove that $h\in\mathbb{K}_{C}[\![\underline{x}]\!][y]$. Hence $f=gh$ with $0<deg_{y}(g)=deg_{y}(G)<n$ and $0<deg_{y}(h)=deg_{y}(H)<n=deg_{y}(f)$, and so $f$ is reducible in $\mathbb{K}_{C}[[\underline{x}]][y]$, which is a contradiction.  Conversely suppose that  $F$ is an irreducible polynomial in $\mathbb{K}[\![\underline{X}]\!][y]$. If $f$ is reducible in $\mathbb{K}_{C}[[\underline{x}]][y]$, then there exist  $h_{1},h_{2}\in\mathbb{K}_{C}[\![\underline{x}]\!][y]$ such that $f=h_{1}h_{2}$ with $0< deg_{y}(h_{1}),deg_{y}(h_{2})<deg_{y}(f)$. Given $a(\underline{x})=\sum c_{\underline{a}}x_{1}^{a_{1}}\cdots x_{e}^{a_{e}}\in \mathbb{K}_{C}[\![\underline{x}]\!]$, we have
$$a(X_{1},X_{2}X_{1},\cdots,X_{e}X_{1})=\sum c_{a}X_{1}^{a_{1}}(X_{2}X_{1})^{a_{2}}\cdots (X_{e}X_{1})^{a_{e}}=
\sum c_{a}X_{1}^{a_{1}+a_{2}+\cdots+a_{e}}X_{2}^{a_{2}}\cdots X_{e}^{a_{e}}$$
Since $a(\underline{x})\in\mathbb{K}_{C}[\![\underline{x}]\!]$, then $a_{1}\geq -(a_{2}+\cdots+a_{e})$ for all $(a_{1},\cdots,a_{e})\in Supp(a(\underline{x}))$. It follows that $a_{1}+a_{2}+\cdots+a_{e}\geq 0$ for all $(a_{1},\cdots,a_{e})\in {\rm Supp}(a(\underline{x}))$. Hence, $a(X_{1},X_{2}X_{1},\cdots,X_{e}X_{1})\in\mathbb{K}[\![\underline{X}]\!]$. Then $h_{1}(X_{1},X_{2}X_{1},\cdots,X_{e}X_{1},y),h_{2}(X_{1},X_{2}X_{1},\cdots,X_{e}X_{1},y)\in\mathbb{K}[\![\underline{X}]\!][y]$. But 
$$F(X_{1},\cdots,X_{e},y)=f(X_{1},X_{2}X_{1},\cdots,X_{e}X_{1},y)=h_{1}(X_{1},X_{2}X_{1},\cdots,X_{e}X_{1},y)h_{2}(X_{1},X_{2}X_{1},\cdots,X_{e}X_{1},y).$$
This contradicts the hypothesis.$\blacksquare$
\end{proof}

\begin{rem} With the notations above, if $f(X_1,X_2X_1,\cdots,X_eX_1,y)$ is irreducible in ${\mathbb K}[\![\underline{X}]\!][y]$ then $f$ is irreducible in ${\mathbb K}[\![\underline{x}]\!][y]$. In fact, if $f$ is reducible in  ${\mathbb K}[\![\underline{x}]\!][y]$ then it is so in 
${\mathbb K}_{C}[\![\underline{x}]\!][y]$. This contradicts Proposition \ref{firreducibleF}. This gives a sufficient irreducibility  criterion in ${\mathbb K}[\![\underline{x}]\!][y]$. This criterion is not a necessary condition. For example, $f=y^2-x_1^2-x_1x_2$ is irreducible in ${\mathbb K}[\![x_1,x_2]\!][y]$, but $f(X_1,X_2X_1,y)=y^2-X_1^2-X_1^2X_2=y^2-X_1^2(1+X_2)=(y-X_1(1+X_2)^{\frac{1}{2}})(y+X_1(1+X_2)^{\frac{1}{2}})$. Note that $f=y^2-x_1^2-x_1x_2=y^2-x_1^2(1+x_1^{-1}x_2)$ is irreducible in ${\mathbb K}_C[\![x_1,x_2]\!][y]$.
\end{rem}

\noindent In the following we give a criterion for the polynomial $f$  to be free.
\begin{prop}\label{Fandf}
Suppose that $f$ is a prepared polynomial. If $f$ is irreducible in $\mathbb{K}_{C}[\![\underline{x}]\!][y]$, then it is  free.
\end{prop}
\begin{proof}{.} By Proposition $\ref{change}$,  $F(X_{1},\cdots,X_{e},y)=f(X_{1},X_{2}X_{1},\cdots, X_{e}X_{1},y)$ is  a quasi-ordinary polynomial of $ \mathbb{K}[\![\underline{X}]\!][y]$, and by Proposition $\ref{firreducibleF}$ we get that $F$ is an irreducible quasi-ordinary polynomial in $\mathbb{K}[\![\underline{X}]\!][y]$ of degree $n$, then by the  Abhyankar-Jung theorem there exists a formal power series $Z$
 in $\mathbb{K}[\![X_{1}^{\frac{1}{n}},\cdots,X_{e}^{\frac{1}{n}}]\!]$ such that $F(\underline{X},Z(\underline{X}))=0$. But $F(\underline{X},Z(\underline{X}))=f(X_{1},X_{2}X_{1},\cdots, X_{e}X_{1},Z(\underline{X}))$,
then $f(x_{1},x_{2},\cdots,x_{e},Z(x_{1},x_{2}x_{1}^{-1},\dots, x_{e}x_{1}^{-1}))=0$. It follows that $Z(x_{1},x_{2}x_{1}^{-1},\dots, x_{e}x_{1}^{-1})$ is a solution of $f(x_{1},\cdots,x_{e},y)=0$. Since $Z(\underline{X})\in\mathbb{K}[\![\underline{X}^{\frac{1}{n}}]\!]$, then by Lemma $\ref{support}$ we deduce that $Z(x_{1},x_{2}x_{1}^{-1},\dots, x_{e}x_{1}^{-1})\in\mathbb{K}_{C}[\![\underline{x}^{\frac{1}{n}}]\!]$. This proves our assertion.$\blacksquare$
\end{proof}

\begin{rem} The criterion of Proposition \ref{Fandf} is effective. In fact, in order to decide if $f$ is irreducible in ${\mathbb K}_C[\![\underline{x}]\!][y]$ one has to decide if $f(X_1,X_2X_1,\cdots,X_eX_1,y)$ is irreducible in  ${\mathbb K}[\![\underline{X}]\!][y]$. But 
$f(X_1,X_2X_1,\cdots,X_eX_1,y)$ is a quasi-ordinary polynomial, hence we can apply the irreducibility criterion given in \cite{A.Assi1}.
\end{rem}

\begin{rem} In Propositions \ref{firreducibleF} and \ref{Fandf}, if $F(\underline{X})=f(X_1,X_2X_1,\cdots,X_eX_1)$ is not irreducible, then it decomposes into quasi-ordinary polynomials, hence $f$ itself decomposes into free polynomials in ${\mathbb K}_C[\![\underline{x}]\!][y]$. As for reducible quasi-ordinary polynomials, we can  associate with $f$ the set of  characteristic sequences of its irreducible components as well as a semigroup defined from the set of semigroups of these components.
\end{rem}

\noindent Next we prove that the approximate roots of a prepared free polynomial with respect to its $d$-sequence are free polynomials 

\begin{prop}\label{19}
Suppose that $f$ is prepared and let $1\leq k\leq r$.  If $f$ is free in $\mathbb{K}_C[\![\underline{x}]\!][y]$ then App$(f,d_k)$  is  also free.
\end{prop}
\begin{proof}{.} By Propositions \ref{change}, \ref{Fandf} and Lemma \ref{firreducibleF}, the polynomial  $F(\underline{X},y)=f(X_{1},X_{2}X_{1},\cdots, X_{e}X_{1},y)$ is  an irreducible quasi-ordinary  polynomial of $ \mathbb{K}[\![\underline{X}]\!][y]$. Let $G_k={\rm App}(F,d_k)$.  We have
$F=G_k^{d_k}+C_{2}(\underline{X},y)G_k^{d_k-2}+\cdots+C_{d_k}(\underline{X},y)$, with deg$_{y}(C_{i})<\frac{n}{d_k}$ for all $i\in\{2,\cdots,d_k\}$. Hence, 
$
f(x_{1},\cdots,x_{e},y)=F(x_{1},x_{2}x_{1}^{-1},\dots ,x_{e}x_{1}^{-1},y)
=g_k^{d}(\underline{x},y)+C_{2}'(\underline{x},y)g_k^{d_k-1}(\underline{x},y)+\cdots+C_{d_k}'(\underline{x},y)
$
where $g_k(\underline{x},y)=G_k(x_{1},x_{2}x_{1}^{-1},$
$\dots ,x_{e}x_{1}^{-1},y)$ and $C_{i}'(\underline{x},y)=C_{i}(x_{1},x_{2}x_{1}^{-1},\dots ,x_{e}x_{1}^{-1},y)$ for all $i\in\{2,\cdots,d\}$. By lemma $\ref{support}$ we have $g_k,C_{i}'\in\mathbb{K}_{C}[\![\underline{x}]\!][y]$ for all $i\in\{2,\cdots,n\}$. Since $deg_{y}(C_{i}')<\frac{n}{d}$ for all $i\in\{2,\cdots,d\}$ and deg$_{y}(g)=\frac{n}{d}$ we get that $g={\rm App}(f,d)$ in $\mathbb{K}_{C}[\![\underline{x}]\!][y]$.  But $f\in\mathbb{K}[\![\underline{x}]\!][y]$ and $\mathbb{K}[\![\underline{x}]\!][y]\subseteq \mathbb{K}_{C}[\![\underline{x}]\!][y]$, then $g={\rm App}(f,d)$ in $\mathbb{K}[\![\underline{x}]\!][y]$. Since $G$ is the approximate root of an irreducible quasi-ordinary polynomial then it is an irreducible quasi-ordinary polynomial, and  $G$ admits a root in $\mathbb{K}[\![\underline{x}^{\frac{1}{\frac{n}{d_k}}}]\!]$. But $g_k(\underline{x},y)=G_k(x_{1},x_{2}x_{1}^{-1},\dots ,x_{e}x_{1}^{-1},y)$, then by a similar argument as in Proposition $\ref{Fandf}$ we get that $g$ admits a root in $\mathbb{K}_{C}[\![\underline{x}^{\frac{1}{\frac{n}{d}}}]\!]$. Moreover $g_k$ is irreducible in $\mathbb{K}_{C}[\![\underline{x}]\!][y]$ by lemma $\ref{firreducibleF}$. Hence $g_k$ is free.$\blacksquare$
\end{proof}

\begin{rem} The result of Proposition \ref{19} is false if we consider any $d^{th}$ approximate root of $f$, even for $e=1$. For a counterexample, see \cite{B}, Theorem 5.

\end{rem}

\begin{ex} Let $f(x_1,x_2,y)=(y^2-x_1^3)^2-4x_1^4x_2y-x_1^5x_ 2^2$. Then $f$ is irreducible in ${\mathbb K}[\![x_1,x_2]\!][y]$ as it is the minimal polynomial of $y_1=x_1^{\frac{6}{4}}+x_1^{\frac{5}{4}}x_2^{\frac{2}{4}}$ over ${\mathbb K}(\!(x_1,x_2)\!)$, and the other solutions are given by $y_2=x_1^{\frac{6}{4}}-x_1^{\frac{5}{4}}x_2^{\frac{2}{4}}$, $y_3=-x_1^{\frac{6}{4}}+x_1^{\frac{5}{4}}x_2^{\frac{2}{4}}$, and $y_4=-x_1^{\frac{6}{4}}-x_1^{\frac{5}{4}}x_2^{\frac{2}{4}}$. From this we can verify that $f$ is prepared. Now  $F(X_1,X_2,y)=f(X_1,X_2X_1,y)=(y^2-X_1^3)^2-4X_1^5X_2y-X_1^7X_2^2$ whose solutions are given by $Y_i=y_i(X_1, X_1X_2)$, hence $F(X_1,X_2,y)$ is an irreducible quasi-ordinary polynomial (the set of characteristic exponents is given by $\lbrace (6,0), (7,2)\rbrace $  and the semigroup is generated by $(4,0),(0,4),(6,0), (14,2)=O(F,y^2-X_1^3)$). It follows that $f$ is irreducible in ${\mathbb K}_C[\![x_1,x_2,y]\!] ($the set of characteristic exponents is given by $\lbrace (6,0), (5,2)\rbrace $ and the semigroup is generated by $(4,0),(0,4),(6,0), (11,2)=O(f,y^2-x_1^3={\rm App}(f,2)))$.
\end{ex}

\noindent Acknowledgments. The authors would like to thank the anonymous referee for his valuable comments on the paper. 

\end{document}